\documentclass[11pt]{article}

\usepackage{amsfonts}
\usepackage{amsthm}
\usepackage{amsmath}
\usepackage{epsfig}
\usepackage{color}
\usepackage{hyperref}
\usepackage[all,tips]{xy}
\usepackage{url,fancyhdr}

\theoremstyle{plain}

\newtheorem{thm}{Theorem}[section]

\newtheorem{prop}[thm]{Proposition}

\newtheorem{lemma}[thm]{Lemma}

\theoremstyle{definition}
\newtheorem*{defn}{Definition}

\newtheorem*{rem}{Remark}

\DeclareMathOperator{\GL}{GL}

\DeclareMathOperator{\PO}{PO}

\newcommand{\ga}{\gamma}
\newcommand{\Ga}{\Gamma}

\newcommand{\La}{\Lambda}

\newcommand{\set}[1]{\left\{#1\right\}}

\newcommand{\B}[1]{\ensuremath{\mathbf{#1}}}

\newcommand{\Cal}[1]{\ensuremath{\mathcal{#1}}}

\newcommand{\Hy}{\ensuremath{\B{H}}}

\newcommand{\Q}{\ensuremath{\B{Q}}}

\newcommand{\Z}{\ensuremath{\B{Z}}}
\newcommand{\C}{\ensuremath{\B{C}}}

\begin{document}
\bibliographystyle{plain}


\title{Collisions at infinity in hyperbolic manifolds}
\author{D.\ B.\ McReynolds\footnote{Partially supported by NSF grant DMS 1105710}\\ Purdue University \and Alan W.\ Reid\footnote{Partially supported by NSF grants.}\\ University of Texas at Austin \and Matthew Stover\footnote{Partially supported by NSF RTG grant DMS 0602191.}\\ University of Michigan}
\maketitle

\begin{abstract}
For a complete, finite volume real hyperbolic $n$-manifold $M$, we investigate the map between homology of the cusps of $M$ and the homology of $M$. Our main result provides a proof of a result required in a recent paper of Frigerio, Lafont, and Sisto. 
\end{abstract}  


\section{Introduction}

Let $M$ be a cusped finite volume hyperbolic $n$-orbifold. Recall that
the thick part of $M$ is the quotient $M_0 = X_M / \pi_1(M)$, where $X_M$ is
the complement in $\Hy^n$ of a maximal $\pi_1(M)$-invariant collection
of horoballs (see for instance \cite{Ratcliffe}). It is known that $M$ and $M_0$ are
homotopy equivalent and $M_0$ is a compact orbifold with boundary
components $E_1, \dots, E_r$. Each $E_j$ is called a \emph{cusp
  cross-section} of $M$. Since horoballs in $\Hy^n$ inherit a natural
Euclidean metric, each cusp cross-section is naturally a flat $(n -
1)$-orbifold. Changing the choice of horoballs preserves the flat structure up to similarity.

The aim of the present note is to provide a proof of a result required in
Frigerio, Lafont, and Sisto \cite{FLS} for their construction in every $n \geq 4$ of
infinitely many $n$-dimensional graph manifolds that do not support
a locally $\mathrm{CAT}(0)$ metric. Specifically,
the following is our principal result.

\begin{thm}\label{thm:homology-death}
For every $n \geq 3$ and $n > k \geq 2$, there exist infinitely many commensurability classes of orientable non-compact finite volume hyperbolic $n$-manifolds $M$ containing a properly embedded totally geodesic hyperbolic $k$-submanifold $N$ with the following properties. Let $\mathcal E = \set{E_1, \dots, E_r}$ be the cusp cross-sections of $M$ and $\mathcal F = \set{F_1, \dots, F_s}$ the cusp cross-sections of $N$. Then:
\begin{itemize}

\item[(1)] $M$ and $N$ both have at least two ends, and all cusp cross-sections are flat tori;

\item[(2)] the inclusion $N \to M$ induces an injection of $H_{k - 1}(\mathcal F; \Q)$ into $H_{k - 1}(\mathcal E; \Q)$;

\item[(3)] the induced homomorphism from $H_{k - 1}(\mathcal F; \Q)$ to $H_{k - 1}(M; \Q)$ is not an injection.
\end{itemize}
\end{thm}

That every finite volume hyperbolic $n$-orbifold has a finite covering for which (1) holds is a folklore result for which we give a complete proof in $\S$\ref{sec:tori} (see Proposition \ref{prop:torus-cusps}). A proof for $k=2$, the case needed in \cite{FLS}, is given in Chapter 12 of \cite{FLS} assuming that (1) is known. However, our proof yields the more general result stated above. The main difficulty in proving Theorem \ref{thm:homology-death} is (2), which will follow from a separation result that requires some terminology.

Consider a totally geodesic hyperbolic $k$-manifold $N$ immersed in
$M$. Assume that $N$ is also noncompact with cusp cross-sections $F_1,
\dots, F_s$. Though they are not freely homotopic in $N$, it is
possible that two distinct ends of $N$ become freely homotopic inside
$M$. When this occurs, we say that the two ends of $N$ \emph{collide
  at infinity} inside $M$. For certain $N$, it is a well-known consequence
of separability properties of $\pi_1(N)$ in $\pi_1(M)$ (e.g., see
\cite{Be} and \cite{Long--Reid}) that one can find a finite covering $M'$ of $M$
into which $N$ embeds. However, $N$ may still have collisions at infinity
in $M'$ and such collisions can lead to the continued failure of (2). Removing these collisions is the content of the next result; see \S 2 for the definition of virtual retractions.

\begin{thm}\label{thm:remove-collisions}
  Suppose $M$ is a cusped finite volume hyperbolic $n$-manifold and
  $N$ is an immersed totally geodesic cusped hyperbolic $k$-manifold. If $\pi_1(M)$ virtually retracts $\pi_1(N)$, then there exists a finite covering $M'$ of $M$ such that $N$
  embeds in $M'$ and has no collisions at infinity.
\end{thm}

That there are infinitely many commensurability classes of
manifolds for which Theorem \ref{thm:remove-collisions} applies is discussed in the remark at the end of \S 2.  In
particular, any noncompact arithmetic hyperbolic $n$-manifold has the required
property, and these manifolds determine infinitely many commensurability classes in every dimension $n>2$. Given that there are infinitely many commensurability classes to which Theorem \ref{thm:remove-collisions} applies, we now assume Theorem \ref{thm:remove-collisions} and (1) and prove (2) and (3) of Theorem \ref{thm:homology-death}.

\begin{proof}[Proof of (2) and (3) of Theorem \ref{thm:homology-death}]
  Let $M$ and $N$ satisfy the conditions of Theorem
  \ref{thm:remove-collisions} and assume that $M$ satisfies (1). We replace $M$ with the covering $M'$
  satisfying the conclusions of Theorem \ref{thm:remove-collisions} and let $\Cal{E}' = \{E_j'\}$ be the set of cusp ends of $M'$. Note that $M'$ also satisfies (1). Now, each $(k - 1)$-torus $F_j$ is realized as an
  embedded homologically essential submanifold of some $(n - 1)$-torus
  $E_{i(j)}'$. In particular, (3) is
immediate as the homology class $\sum [F_j]$ bounds the class
$[N_0] = [N \cap M'_0]$ inside $M'$. Since $F_j$ is not freely homotopic to $F_k$ in $M'$
for any $j \neq k$ , it follows that they cannot be freely homotopic
in $\mathcal E'$ for any $j \neq k$. Hence the induced map on
$(k - 1)$-homology is an injection, which gives (2). This completes the proof.
\end{proof}

\noindent{\bf Acknowledgements:}~We thank Jean Lafont for stimulating conversations and for drawing our attention to the basic questions addressed in this article.


\section{The proof of Theorem \ref{thm:remove-collisions}}\label{sec:main-thm}

The proof of Theorem \ref{thm:remove-collisions} is an easy consequence of the \emph {virtual retract property}  of \cite{LRRetract} (see also \cite{BHW}) which has found significant applications in low-dimensional topology and geometric group theory of late (see \cite{LRRetract}, \cite{BHW} and the references therein).

\begin{defn}
  Let $G$ be a group and $H < G$ be a subgroup. Then $G$
\emph{virtually retracts} onto $H$ if there exists a finite index subgroup
  $G' < G$ with $H < G'$ and a homomorphism $\rho\colon G' \to H$ such that
  $\rho|_H = \mathrm{id}_H$. In addition we say that $G'$ \emph{retracts}
onto $H$, and $\rho$ is called the \emph{retraction homomorphism}. \end{defn}

With this definition we note the following lemma.

\begin{lemma}\label{lem:retract-conjugacy}
Let $G$ be a group and $H < G$ a subgroup such that $G$ retracts onto $H$. Then two subsets $S_1, S_2$ of $H$ are conjugate in $G$ if and only if they are conjugate in $H$.
\end{lemma}

\begin{proof}
One direction is trivial. Suppose that there exists $g \in G$ such that $S_1 = g S_2 g^{-1}$. Then
\[
S_1 = \rho(S_1) = \rho(g S_2 g^{-1}) = \rho(g) S_2 \rho(g)^{-1},
\]
so $S_1$ and $S_2$ are conjugate in $H$.
\end{proof}

\begin{proof}[Proof of Theorem \ref{thm:remove-collisions}]
  Let $M = \Hy^n / \Ga$ be a cusped finite volume hyperbolic
  $n$-manifold, $N = \Hy^k / \La$ be a noncompact finite volume
  totally geodesic hyperbolic $k$-manifold immersed in $M$ such that
  that $\Ga$ virtually retracts onto $\La$. Let $F_1, \dots, F_r$ be
  the cusp cross-sections of $N$ and $\Delta_1, \dots, \Delta_r < \La$
  representatives for the associated $\La$-conjugacy classes of
  peripheral subgroups, i.e., $\Delta_j = \pi_1(F_j)$.

Two ends $F_{j_1}$ and $F_{j_2}$ of $N$ collide at infinity in $M$ if
and only if any two representatives $\Delta_{j_1}$ and $\Delta_{j_2}$
for the associated $\La$-conjugacy classes of peripheral subgroups are
conjugate in $\Ga$ but not in $\La$. Let $\Ga_N$ denote the finite
index subgroup of $\Ga$ that retracts onto $\La$, and $\rho\colon \Ga_N
\to \La$ the retracting homomorphism. By Lemma \ref{lem:retract-conjugacy}, $\Delta_{j_1}$ and
$\Delta_{j_2}$ are not conjugate in $\Ga_N$ for any $j_1 \neq
j_2$. Thus $N$ has no collisions at infinity inside $M' = \Hy^n /
\Ga_N$.

Moreover, since $\La$ is a retract of $\Ga_N$, it follows that $\La$ is
separable in $\Ga_N$ (see Lemma 9.2 of \cite{HW}). Now a well-known result
of Scott \cite{Sc} shows that we can pass to a further covering $M''$ of $M'$ such
that the immersion of $N$ into $M'$ lifts to an embedding in $M''$. This proves
the theorem.
\end{proof}

\begin{rem}\
\begin{enumerate}

\item Examples where the virtual retract property holds are
  abundant. From \cite{BHW}, if $M=\Hy^n/\Gamma$ is any non-compact
  finite volume hyperbolic $n$-manifold, which is arithmetic or arises
  from the construction of Gromov--Piatetskii-Shapiro, then $\Gamma$
  has the required virtual retract property. Briefly, the arithmetic case follows from Theorem 1.4 of \cite{BHW}
  and the discussion at the very end of \S 9 of \cite{BHW}, and for the
  examples from the Gromov--Piatetskii-Shapiro construction it follows
  from Theorem 9.1 of \cite{BHW} and the same discussion at the very
  end of in \S 9.

\item We have in fact shown something stronger, namely that two essential loops in a cusp cross-section $F_j$ of $N$ are homotopic inside $M'$ if and only if they are freely homotopic in $N$. Therefore, the kernel of the induced map from $H_*(\mathcal F; \Q)$ to $H_*(M'; \Q)$ is precisely equal to the kernel of the homomorphism from $H_*(\mathcal F; \Q)$ to $H_*(N; \Q)$.

\item Lemma \ref{lem:retract-conjugacy} also implies that $N$ cannot have positive-dimensional essential self-intersections inside $M'$. In particular, if $n < 2 k$, then $N$ automatically embeds in $M'$.

\end{enumerate}
\end{rem}


\section{Covers with torus ends}\label{sec:tori}

The following will complete the proof of Theorem \ref{thm:homology-death}.

\begin{prop}\label{prop:torus-cusps}
  Let $M$ be a complete finite volume cusped hyperbolic
  $n$-manifold. Then $M$ has a finite covering $M'$ such that $M'$ has
  at least two ends and each cusp cross-section is a flat $(n -
  1)$-torus.
\end{prop}

\begin{proof}[Proof of Theorem \ref{prop:torus-cusps}]
Let $M = \Hy^n/\Gamma$ be a cusped hyperbolic $n$-manifold of finite volume. Let $\Delta_1,\dots,\Delta_{r_j}$ be representatives for the conjugacy classes of peripheral subgroups of $\Ga$. For each $\Delta_j$ the Bieberbach Theorem \cite[$\S$7.4]{Ratcliffe} gives a short exact sequence
\[ 1 \to \Z^{n - 1} \to \Delta_j \to \Theta_j \to 1 \]
where $\Theta_j$ is finite. Then $E_j$ is a flat $(n - 1)$-torus if and only if $\Theta_j$ is the trivial group. Note in the case when $M$ is a surface, the statement is trivial and thus we will assume $n>2$.

Let $\gamma_{j,1}, \dots, \gamma_{j,r_j}$ be lifts of the distinct nontrivial elements of $\Theta_j$ to $\Delta_j$.  Since $n \geq 3$, it is a well-known consequence of Weil Local Rigidity (see \cite[Thm.~7.67]{Raghunathan}) that we can conjugate $\Ga$ in $\PO_0(n, 1)$ inside $\GL_N(\C)$ so that it has entries in some number field $k$. Since $\Ga$ is finitely generated, we can further assume that it has entries in some finitely generated subring $R \subset k$. Then $R / \mathfrak p$ is finite for every prime ideal $\mathfrak p \subset R$.

This determines a homomorphism from $\Ga$ to $\GL_N(R / \mathfrak p)$. For every $\gamma_{j, k}$, the image of $\gamma_{j, k}$ in the finite group $\GL_N(R / \mathfrak p)$ is nontrivial for almost every prime ideal $\mathfrak p$ of $R$. Indeed, any off-diagonal element is congruent to zero modulo $\mathfrak p$ for only finitely many $\mathfrak p$ and there are only finitely many $\mathfrak p$ so that a diagonal element is congruent to $1$ modulo $\mathfrak p$. Since there are finitely many $\gamma_{j, k}$, this determines a finite list of prime ideals $\mathcal P = \set{\mathfrak p_1, \dots, \mathfrak p_s}$ such that $\gamma_{j, k}$ has nontrivial image in $\GL_N(R / \mathfrak p)$ for any $\mathfrak p \notin \mathcal P$ and every $j, k$. If $\Ga(\mathfrak p)$ is the kernel of this homomorphism, then $\Ga(\mathfrak p)$ contains no conjugate of any of the $\gamma_{j, k}$.

The peripheral subgroups of $\Ga(\mathfrak p)$ are all of the form $\Ga(\mathfrak p) \cap \ga \Delta_j \ga^{-1}$ for some $\ga \in
\Ga$. Since no conjugate of any $\ga_{j, k}$ is contained in
$\Ga(\mathfrak p)$, we see that $\Ga(\mathfrak p) \cap \ga \Delta_j
\ga^{-1}$ is contained in the kernel of the above homomorphism $\ga
\Delta_j \ga^{-1} \to \Theta_j$. It follows that every cusp
cross-section of $\Hy^n / \Ga(\mathfrak p)$ is a flat torus. This
proves the second part of the theorem.

To complete the proof of Proposition \ref{prop:torus-cusps}, it
suffices to show that if $M$ is a noncompact hyperbolic $n$-manifold
with $k$ ends, then $M$ has a finite sheeted covering $M'$ with
strictly more than $k$ ends. We recall the following elementary fact
from covering space theory. Let $\rho\colon \Ga \to Q$ be a homomorphism
of $\Ga$ onto a finite group $Q$ and $\Ga_\rho$ be the kernel of
$\rho$. If $\Delta_j$ is a peripheral subgroup of $\Ga$, then the
number of ends of $\Hy^n / \Ga_\rho$ covering the associated end of
$\Hy^n / \Ga$ equals the index $[Q : \rho(\Delta_j)]$ of
$\rho(\Delta_j)$ in $Q$. Therefore, it suffices to find a finite
quotient $Q$ of $\Ga$ and a peripheral subgroup $\Delta_j$ of $\Ga$
that $\rho(\Delta_j)$ is a proper subgroup of $Q$.

In our setting, the proof is elementary. From above, we
can pass to a finite sheeted covering of $M$, for which all the cusp cross-sections are tori,
i.e., all peripheral subgroups are abelian. It follows that for
$\rho|_{\Delta_j}$ to be onto, $\rho(\Ga)$ must be abelian. However,
it is well-known that the above reduction quotients $\Ga /
\Ga(\mathfrak p)$ are central extensions of non-abelian finite simple
groups for all but finitely many prime ideals $\mathfrak p$
\cite[Chapter 6]{Lubotzky--Segal}. The theorem follows.
\end{proof}

\begin{rem}
 Constructing examples with a small number of ends is much more difficult. For example, there are no known one-cusped hyperbolic $n$-orbifolds for $n > 11$. Furthermore, it is shown in \cite{Stover} that for every $d$, there is a constant $c_d$ such that $d$-cusped \emph{arithmetic} hyperbolic $n$-orbifolds do not exist for $n > c_d$. For example, in the case $d = 1$, there are no $1$-cusped arithmetic hyperbolic $n$-orbifolds for any $n \geq 30$.
\end{rem}



\end{document}